\documentclass[11pt]{article}
\usepackage{amsmath,amsthm,amsfonts}
\usepackage{fullpage}
\usepackage[charter]{mathdesign}
\usepackage{dsfont}
\usepackage[colorlinks=true]{hyperref}
\usepackage{microtype}

\newcommand{\supp}{\operatorname{supp}}
\newcommand{\C}{\mathbb{C}}
\renewcommand{\vec}[1]{#1}
\DeclareMathOperator{\op}{op}

\DeclareMathOperator{\tr}{tr}

\DeclareMathOperator{\Res}{Res}

\DeclareMathAlphabet{\varmathbb}{U}{bbold}{m}{n}
\newcommand{\one}{\mathds{1}}
\DeclareMathOperator*{\Exp}{\mathbb{E}}

\DeclareMathOperator{\eee}{\mathrm{e}}
\DeclareMathOperator{\fS}{\operatorname{S}}

\begin{document}

\newtheorem{theorem}{Theorem}
\newtheorem{lemma}{Lemma}
\newtheorem{claim}{Claim}
\newtheorem{proposition}{Proposition}
\newtheorem{exercise}{Exercise}
\newtheorem{corollary}{Corollary}
\theoremstyle{definition}
\newtheorem{definition}{Definition}

\newcommand{\eps}{\epsilon}
\newcommand{\Z}{{\mathbb{Z}}}
\newcommand{\F}{{\mathbb{F}}}
\newcommand{\Var}{\mathop\mathrm{Var}}
\newcommand{\Cov}{\mathop\mathrm{Cov}}
\renewcommand{\Re}{\mathop\mathrm{Re}}
\newcommand{\frob}[1]{\left\| #1 \right\|_{\rm F}}
\newcommand{\opnorm}[1]{\left\| #1 \right\|_{\rm op}}
\newcommand{\bopnorm}[1]{\bigl\| #1 \bigr\|_{\rm op}}
\newcommand{\norm}[1]{\left\| #1 \right\|}
\newcommand{\bnorm}[1]{\bigl\| #1 \bigr\|}
\newcommand{\abs}[1]{\left| #1 \right|}
\newcommand{\inner}[2]{\left\langle #1, #2 \right\rangle}

\newcommand{\wf}{\widehat{f}}
\newcommand{\wg}{\widehat{g}}
\newcommand{\wG}{\widehat{G}}
\newcommand{\wK}{\widehat{K}}
\newcommand{\wH}{\widehat{H}}
\newcommand{\wpsi}{\widehat{\psi}}
\newcommand{\wpsidag}{\widehat{\psi^\dagger}}
\newcommand{\U}{\textsf{U}}
\newcommand{\GL}{\textsf{GL}}
\newcommand{\SL}{\mathsf{SL}}
\newcommand{\fsub}{\operatorname{F}}
\newcommand{\dmin}{d_{\min}}
\newcommand{\planch}{\mathcal{P}}
\newcommand{\frack}[2]{#1 / #2}

\title{Group representations that resist random sampling}

\author{Shachar Lovett\\Computer Science and Engineering\\University of California, San Diego\\ \texttt{slovett@cs.ucsd.edu}  \and Cristopher Moore\\Santa Fe Institute\\{\texttt {moore@cs.unm.edu}} 
\and Alexander Russell\\Computer Science and Engineering\\University of Connecticut\\{\texttt{acr@cse.uconn.edu}}}

\maketitle

\begin{abstract}
We show that there exists a family of groups $G_n$ and nontrivial irreducible representations
$\rho_n$ such that, for any constant $t$, the
average of $\rho_n$ over $t$ uniformly random elements $g_1, \ldots,
g_t \in G_n$ has operator norm $1$ with probability approaching 1 as
$n \rightarrow \infty$. 
More quantitatively, we show that there exist families of finite
groups for which $\Omega(\log \log |G|)$ random elements are
required to bound the norm 
of a typical representation below $1$. This settles a conjecture of A.\ Wigderson. 
\end{abstract}

\section{Introduction}
The Alon-Roichman theorem~\cite{AR94} asserts that $O(\log
|G|/\epsilon^2)$ elements, chosen independently and uniformly from a
finite group $G$, yield with high probability a Cayley graph with
second eigenvalue no more than $\epsilon$. As there are groups for
which $O(\log |G|)$ elements are necessary to even generate the group,
this bound is tight up to a constant when $\epsilon = \Omega(1)$. The
condition that a collection of group elements $g_1, \ldots, g_t$ 
yield a graph with second eigenvalue $\epsilon$ is
equivalent to the condition that every nontrivial irreducible representation
$\rho$ of $G$ is approximately annihilated in the sense that
\begin{equation}
\label{eq:norm}
\left \| \frac{1}{t} \sum_{i=1}^t \frac{\rho(g_i) + \rho(g_i^{-1})}{2} \right\|_{\op} \leq \epsilon \, .
\end{equation}
(The appearance of the inverses $g_i^{-1}$ above corresponds to the
constraint that the Cayley graph be undirected.)

This invites a more refined analysis of the Alon-Roichman theorem
popularized by a question of A.\ Wigderson~\cite[Conjecture 2.8.4]{Barbados10}: \emph{Are there universal constants
$t$ and $\delta > 0$ such that
\begin{equation}
\label{eq:average-norm}
\Exp_{g_1,\ldots,g_t} \left[ \opnorm{\frac{1}{t} \sum_{i=1}^t \rho(g_i) } \right] \leq 1 - \delta
\end{equation}
 holds for all finite
groups $G$ and all nontrivial irreducible representations $\rho$ of
$G$?} As above, the $g_i$ are selected independently and uniformly
from $G$.

To support this notion, we remark that there are known families of ``highly
nonabelian'' finite groups, such as $\SL_2(\F_p)$ as $p \rightarrow
\infty$, which yield expanding Cayley graphs over a constant number of
uniformly random group elements $g_1,\ldots,g_t$~\cite{BG08}. It
follows that~\eqref{eq:average-norm} holds for every representation.
Likewise, \eqref{eq:average-norm} holds for all nontrivial irreducible representations of abelian
groups: for example, while no constant number of elements suffice to
make $\Z_2^n$ into an expander, or even to generate the group, just
$t=2$ uniformly random elements suffice to bound the expected norm of any one
irreducible representation to $1-\delta$, with a universal constant $\delta > 0$.

Distinct irreducible representations of a finite group possess various
independence properties when viewed as random variables by selecting a
uniformly random group element: for example, their entries are
pairwise uncorrelated in any basis. Intuitively, Wigderson's question
asks whether the dependence on $\log |G|$ in the Alon-Roichman theorem is an
artifact of the requirement that every representation be
annihilated---a collection of ostensibly independent events which one
might imagine occur with constant probability with the selection of
each new group element---or is a feature that can manifest even in a
single irreducible representation. Indeed, a positive answer to the
question would imply that $O(\log |\wG|)$ random elements suffice to
turn any group $G$ into an expanding Cayley graph, where $\hat{G}$ denotes the
set of irreducible representations of $G$.


We answer this question in the negative. Our strategy will be to work
in a family of finite groups of the form $G=K^n$, where $K$ has constant size, is nonabelian, and has trivial center.  
Such groups simultaneously possess high-dimensional representations and the
property that any collection of a bounded number of group elements
generate a subgroup of bounded size. In this setting, for each
constant $t > 0$, we establish two results:
\begin{itemize}
\item When $K$ has a faithful irreducible representation $\rho$ of dimension at least
two, we show that the representation $\rho^n = \rho \otimes \cdots \otimes
\rho$ of $G = K^n$ has the property that
\[
\Pr_{g_1, \ldots, g_t}\left[ \left\| \frac{1}{t} \sum_i \rho^n(g_i) \right\|_{\op} =
1\right] \geq 1 - \exp(-\Omega(n)) \, ,
\]
where the $g_i$ are elements of $G$ chosen uniformly and independently
at random. See Theorem~\ref{thm:main-explicit} for a more precise statement.
\end{itemize}
This result holds, for instance, if $K=S_3$, the group of permutations of three elements, 
and $\rho$ is its two-dimensional representation. 
We also show
\begin{itemize}
\item When $\rho$ is selected according to the Plancherel measure, which
  assigns each irreducible representation probability mass proportional to the
  square of its dimension, we show that with probability $1 -
  O(1/\sqrt[4]{n})$, $\rho$ has the property that
\[
\Pr_{g_1, \ldots, g_t}\left[ \left\| \frac{1}{t} \sum_i \rho^n(g_i) \right\|_{\op} =
1\right] \geq 1 - O\left(\frac{1}{\sqrt[4]{n}}\right) \, ,
\]
where the $g_i$ are elements of $G$ chosen uniformly and independently
at random. See Theorem~\ref{thm:main} for a more precise statement.
\end{itemize}

In fact, these estimates establish that there are infinite families of
groups for which there is a representation $\rho$ such that 
$\Omega(\log \log |G|)$ uniformly random elements are necessary to bound the 
norm of $\rho$ as in~\eqref{eq:average-norm}, for any constant $\delta > 0$.  
Recall that the Alon-Roichman theorem guarantees that
$O(\log |G|)$ random elements suffice with high probability to bound
the norm of \emph{every} irrep $\rho$, and thus turn $G$ into an
expander. Closing this gap remains an interesting open question.

These negative results for independent and uniformly random group elements
suggest the following question, which is existential rather than probabilistic: 
Are there constants $t$ and $\delta > 0$ such that, for any group $G$ and any nontrivial irreducible representation $\rho$, 
there exist $t$ elements $g_1, \ldots, g_t \in G$ such that
\[
\left\| \frac{1}{t} \sum_i \rho(g_i)\right\|_{\rm op} \leq 1-\delta\,?
\]
Our construction cannot rule out this possibility.

\subsection*{Notation and actions on the group algebra} 
Given a finite group $G$, let $\wG$ denote its set of irreducible
representations. 
We assume throughout that all representations are unitary.  
For a representation $\rho$, we let $\chi_\rho(g) =
\tr \rho(g)$ denote its character, and let $d_\rho = \chi_\rho(1)$
denote its dimension. For two class functions, $\chi$ and $\psi$ on $G$,
we define 
\begin{equation}\label{eq:ip}
\langle \chi, \psi\rangle_G = \frac{1}{|G|} \sum_g \chi(g)\psi(g)^*
\end{equation}
and remark that the characters of the irreducible representations of
$G$ form an orthonormal basis for the space of class functions with
respect to the inner product~\eqref{eq:ip}.  Thus if $\rho$ is irreducible and $\sigma$ is a representation, 
$\langle \rho , \sigma \rangle_G$ is the number of copies of $\rho$ appearing in $\sigma$.

If $H$ is a subgroup of $G$ and $\chi$ is the character of a
representation of $G$, we let $\Res_H \chi$ denote this class function
restricted to the subgroup $H$; when we wish to emphasize the ambient
group $G$, we write $\Res^G_H \chi$. If $G=K^n$ and $\rho$ is a representation of $K$,
we let $\rho^n$ denote the representation of $G$ given by the rule
\[
\rho^n(g_1, \ldots, g_n) = \rho(g_1) \otimes \cdots \otimes \rho(g_n)
\]
and remark that $\chi_{\rho^n}(g_1, \ldots, g_n) = \prod_{j=1}^n
\chi_\rho(g_j)$. We overload this notation, defining 
\[
\chi_{\rho^n}(g)
= \Res_K \chi_{\rho^n} (g, \ldots, g)
\]
to be the character obtained on $K$ by
restricting $\chi_{\rho^n}$ to the ``diagonal'' subgroup 
$\{ (g, \ldots, g) \mid g \in K\} \cong K$.

Let $\C[G]$ denote the group algebra on $G$. The left action of $G$ on
$\C[G]$ obtained by linearly extending the rule $g: g' \mapsto g g'$
induces the \emph{(left) regular representation} $R$ of $G$, with character
$$
\chi_R(g) = \begin{cases} |G| & \text{if $g = 1$} \, ,\\
  0 & \text{otherwise.} \end{cases}
$$
As a consequence of character orthogonality, one can express $R$ as the sum
\[
R = \bigoplus_{\rho \in \wG} d_\rho \rho \, , 
\]
i.e., it contains $d_\rho$ copies of each irreducible representation $\rho$.  Thus its character can be written
$$
\chi_R = \sum_{\rho \in \wG} d_\rho \chi_\rho  \, . 
$$
The algebra $\C[G]$ can likewise be given the structure of a $G \times G$ representation $B$ by linearly extending the rule $(g_1, g_2): g \mapsto g_1 g g_2^{-1}$.  
Its character $\chi_B$ thus satisfies
$$
\chi_B(g_1, g_2) = \abs{ \left\{ g \,:\, g_1 g g_2^{-1} = g \right\} }  \, .
$$
Again applying character orthogonality, $B$ can be expressed as the sum 
\[
B = \bigoplus_{\rho \in \wG} \rho \otimes \rho^*
\]
and thus
$$
\chi_B(g_1, g_2) = \sum_{\rho \in \wG} \chi_\rho(g_1) \cdot \chi_\rho^*(g_2)  \, .
$$

Finally, the Plancherel measure $\planch$ on $\wG$ assigns each representation $\rho$ the probability mass
\[
\planch(\rho) = \frac{d_\rho^2}{|G|}  \, . 
\]
If $G=K^n$, then selecting $\rho = \rho_1 \otimes \cdots \otimes \rho_n$ from the Plancherel measure on $\wG$ is the same as selecting the $\rho_i$ independently from the Plancherel measure on $\wK$.

\section{Remarks on the subgroup structure of groups of the form $K^n$}
Anticipating the proofs, we collect a few facts about subgroups of $G=K^n$.  
Given $g \in G$, let $g_i \in K$ denote its $i$th coordinate; note that $\pi_i(g)=g_i$ is a homomorphism from $G$ to $K$.
We consider the subgroup $H$ generated by a collection of $t$ elements $h^{(1)}, \ldots, h^{(t)}$ of $K^n$. 
Given such a collection, 
define the function $\fS: \{ 1, \ldots, n\} \rightarrow K^t$ so that $S(i)$ lists their $i$th coordinates:
\begin{equation}
\label{eq:def-s}
\fS(i) = ( h^{(1)}_i, \ldots, h^{(t)}_i) \, . 
\end{equation}
Let $\sim_{\fS}$ denote the equivalence class on $\{1,\ldots,n\}$ given by the
level sets of $\fS$, so that 
\[ 
i \sim_{\fS} j \Leftrightarrow \fS(i) = \fS(j) \, . 
\] 
That is, $i \sim_{\fS} j$ if and only if $h^{(m)}_i = h^{(m)}_j$ for all $1 \le m \le t$.  
As a consequence, we also have $h_i = h_j$ for any $h \in H$.  
Thus if we define the subgroup of elements of $K^n$ that respect this equivalence, 
\begin{equation}\label{eq:def-H-tilde}
\tilde{H} = \{ (g_1, \ldots, g_n) \mid i \sim_{\fS} j \Rightarrow g_i = g_j\} \subseteq K^n \, ,
\end{equation}
we have $h^{(m)} \in \tilde{H}$ for all $m$, and hence $H \subset \tilde{H}$.  
(Note that $S$, $\sim_S$, and $\tilde{H}$ depend implicitly on the collection $h^{(1)}, \ldots, h^{(t)}$.)

While the group $H$ may have quite complicated structure, $\tilde{H}$ is isomorphic
to $K^\ell$ for some $\ell \geq 1$, where $\ell \le n$ is the number of equivalence classes of $\sim_{\fS}$.  
Moreover, as the function $\fS$ takes no more than $|K|^t$
different values, $\ell \leq |K|^t$ and  
$|H| \leq |\tilde{H}| \leq |K|^{|K|^t}$.  
Thus, as recorded in the following lemma, if $|K|$ and $t$ are constant then $H$ is of constant size.
\begin{lemma}
  \label{lem:subgroup-size}
  Let $H$ be the subgroup of $K^n$ generated by $t$ elements $h^{(1)}, \ldots, h^{(t)}$. Then $|H| \leq |K|^{|K|^t}$.
\end{lemma}

The size of the smallest equivalence class of $\sim_{\fS}$ plays an
essential role in both proofs. To name this quantity, for a sequence of elements $h^{(1)}, \ldots, h^{(t)}$, let
$$
d = d(h^{(1)}, \ldots, h^{(t)}) = \min_{i \in \{1, \ldots, n\}}
\bigl|\{ j \mid i \sim_{\fS} j\}\bigr| \, .
$$
We set down two straightforward tail bounds on
$d(h^{(1)},\ldots,h^{(t)})$, when the $h^{(m)}$ are selected
independently and uniformly at random from $K^n$.
\begin{lemma}
\label{lem:support}
Let $h^{(1)}, \ldots, h^{(t)}$ be $t$ independent and uniformly random elements of
$K^n$. Then
\begin{enumerate}
\item\label{eq:support-large}
  $\displaystyle\Pr\left[ d(h^{(1)}, \ldots, h^{(t)}) \leq \frac{n}{2 |K|^t}
  \right] \leq \frac{4 |K|^{2t}}{n}$
and, 
\item\label{eq:support-small} for any $\ell \geq 1$, 
$\displaystyle
  \Pr\left[ d(h^{(1)}, \ldots, h^{(t)}) < \ell \right] 
  \leq n^\ell |K|^t \eee^{-n/|K|^t} \, .$
\end{enumerate}
\end{lemma}
\begin{proof}
  For a given $\vec{s} \in K^t$, let $X_{\vec{s}}$ be the random variable equal to the size of the corresponding set in the partition, $\{ i : \vec{S}(i) = \vec{s} \}$.  Since $X_{\vec{s}}$ is binomially distributed as $\textrm{Bin}(n,1/|K|^t)$, we have $\Exp X_{\vec{s}} = n/{|K|^t}$ and 
  $\Var X_{\vec{s}} \le n/|K|^t$.  By Chebyshev's inequality, 
  $$
  \Pr\left[ \left| X_{\vec{s}} - \frac{n}{|K|^t}\right| 
  \geq \frac{n}{2 |K|^t} \right] \leq \frac{4|K|^t}{n}  \, .
  $$
Since there are $|K|^t$ elements of $K^t$, the union bound implies the statement~\eqref{eq:support-large} of the lemma. As for statement~\eqref{eq:support-small}, for any $\vec{s}$ the probability that $X_{\vec{s}} < \ell$ is no more than
\[
\sum_{i = 0}^{\ell-1} \binom{n}{i} \left(\frac{1}{|K|^t}\right)^i \left(1 - \frac{1}{|K|^t}\right)^{n-i} 
\leq \sum_{i = 0}^{\ell-1} \binom{n}{i} \left(1 - \frac{1}{|K|^t}\right)^{n} 
\leq \left[\sum_{i=0}^{\ell-1} \binom{n}{i}\right] \left(\eee^{-{1}/{|K|^t}}\right)^n 
\;\leq  n^{\ell} \eee^{-{n}/{|K|^t}} \, .
\]
Then the union bound implies the statement~\eqref{eq:support-small} of
the lemma.
\end{proof}
\noindent
Note that the random variable $X_{\vec{s}}$ obeys the Chernoff bound,
making it much more concentrated than the second moment calculation
of~\eqref{eq:support-large} suggests, but this is not important to our results.

\section{An explicit construction}

\begin{theorem}\label{thm:main-explicit}
Let $G=K^n$ where $K$ is a finite nonabelian group with trivial center.  Let $\rho$ a faithful irreducible representation of $K$ of dimension $d_\rho \geq 2$. Then there is an integer $\kappa_\rho \geq 2$ such that for every $t > 0$,
  \[
  \Pr_{h^{(1)}, \ldots, h^{(t)}} \left[ \left\|\frac{1}{t} \sum_{m=1}^t
      \rho^n(h^{(m)})\right\|_{\op} = 1\right] \geq 1 - n^{\kappa_\rho} |K|^t \exp(-n/|K|^t)  \, ,
  \]
where $h^{(1)}, \ldots, h^{(t)} \in G$ are independent and uniform.
\end{theorem}
\noindent


We begin by observing that a representation satisfying the conditions
of the theorem above has the property that 
\[
\rho^{\otimes k} = \underbrace{\rho \otimes \cdots \otimes
  \rho}_{\textrm{$k$ times}}
\]
contains a copy of the trivial representation for all sufficiently large $k$.  
For instance, if $K=S_3$ and $\rho$ is its two-dimensional representation, then this holds for any $k \ge 2$.  
The following is a classic fact of representation theory, but we give a proof for completeness:
\begin{lemma}
\label{lem:faithful}
  Let $K$ be a finite group with trivial center and $\rho$ a faithful
  irreducible representation of $K$ of dimension $d_\rho \geq 2$. Then there is an integer $\kappa_\rho \geq 2$ so
  that for all $k \geq \kappa_\rho$, $\langle \chi_\rho^{k}, 1\rangle_K > 0$.
\end{lemma}

\begin{proof}
  As $\rho$ is faithful, 
  the only element $h \in K$ such that $\rho(k)$ is a scalar matrix is the identity:
  \[
  \{ h \mid \text{$\rho(h) = \lambda\one$ for some $\lambda \in \C$}\}
  = \mathcal{Z}(K) = \{1\} \, ,
  \]
  where $\mathcal{Z}(K)$ denotes the (trivial) center of $K$.
  By unitarity, it follows that $| \chi_\rho(h) | < d_\rho$ for all $h \neq 1$.
  Expanding
  \[
  \langle \chi_\rho^k, 1 \rangle_K = \frac{1}{|K|} \sum_{h \in K}
  \chi_\rho(h)^k = \frac{d_\rho^k}{|K|}\Biggl[1 + \underbrace{\sum_{h \neq 1}
    \left(\frac{\chi_\rho(h)}{d_\rho}\right)^k}_{(\dagger)}\Biggr] \, ,
  \]
  it is evident that, for sufficiently large $k$, each term of the sum
  $(\dagger)$ is strictly less than $1/(|K|-1)$ in absolute value. For
  such $k$, the quantity in brackets above is strictly positive, as desired.
\end{proof}

\begin{proof}[Proof of Theorem~\ref{thm:main-explicit}]
In light of Lemma~\ref{lem:faithful}, let $K$ be a finite group with trivial
center and $\rho$ a faithful, irreducible representation of $K$ of
dimension $d_\rho > 2$. Consider now the representation $\rho^n$ of the group
$K^n$. For a sequence of elements
$h^{(1)}, \ldots, h^{(t)} \in K^n$, recall that the subgroup
$\tilde{H}$, defined in~\eqref{eq:def-H-tilde} above, contains all
$h^{(m)}$ and that
\[
\Res_{\tilde{H}} \rho^n 
= \bigotimes_{\substack{\text{$C$: equivalence}\\\text{class of $\sim_{\fS}$}}} 
\Res_K^{K^{|C|}} \rho^{|C|} \, .
\]
If each equivalence class $C$ of $\sim_{\fS}$ has size at
least $\kappa_\rho$, each factor in the tensor product above has a
copy of the trivial representation, and thus 
$\Res_{\tilde{H}} \rho^n$ has a copy of the trivial representation.  In that case,
\[
\Bigl\| \frac{1}{t} \sum_{m=1}^t \rho^n(h^{(m)}) \Bigr\|_{\op} = 1 \, .
\]
As each equivalence class of $\sim_{\fS}$ has size at least
$d(h^{(1)}, \ldots, h^{(t)})$, we conclude that
\[
\Pr\left[ \left\| \frac{1}{t} \sum_{m=1}^t \rho^n(h^{(m)}) \right\|_{\op} =1\right] 
\geq 1 - \Pr\left[ d(h^{(1)}, \ldots, h^{(t)}) < \kappa_\rho\right] 
\geq 1 - n^{\kappa_\rho} |K|^t \eee^{-n/|K|^t} \, ,
\]
by statement 2 of Lemma~\ref{lem:support}.
\end{proof}

\section{The behavior of random representations}

Focusing on the same family of groups $G=K^n$ where $K$ is nonabelian with a trivial center, we establish that a representation $\rho$ selected according to the Plancherel measure has the property, with high probability, that
\[
\Pr_{h^{(1)}, \ldots, h^{(t)}}\left[ \left\| \frac{1}{t} \sum_{m=1}^t \rho^n(h^{(m)}) \right\|_{\op} = 1\right] 
\geq 1 - O\left(\frac{1}{\sqrt[4]{n}}\right) \, ,
\]
where $h^{(1)}, \ldots, h^{(t)} \in G$ are independent and uniform.

Our proof focuses on the random variable 
\begin{equation}
\label{eq:xh}
X_H = \frac{\langle \Res_H \chi_\rho, 1\rangle_H}{d_\rho} \, ,
\end{equation}
where $H$ is a fixed subgroup of a group $G$ and $\rho$ is selected
according to the Plancherel measure.  Since $X_H$ is
the dimensionwise fraction of $\rho$ that restricts to the trivial
representation under $H$, whenever $X_H > 0$ we have 
\[
\opnorm{ \Exp_{h \in H} \rho(h) } = 1  \, . 
\]
The proof will show that for groups of the form $K^n$, if $H=\langle h^{(1)}, \ldots, h^{(t)} \rangle$ then $X_H > 0$ with high probability.  We do this by computing the first two moments of $X_H$, first for general group-subgroup pairs, and then specializing to the groups $K^n$.
 
\paragraph{The expectation}  For the first moment, observe that if $\rho \in \wG$ is distributed according to the Plancherel measure, then 
\begin{align*}
  \Exp_\rho X_H 
  &= \Exp_\rho \frac{\langle \Res_H \chi_\rho, 1\rangle_H}{d_\rho} 
  = \sum_\rho \frac{d_\rho^2}{|G|} \frac{\langle \Res_H \chi_\rho, 1\rangle_H}{d_\rho} 
  = \frac{1}{|G|} \sum_\rho d_\rho \langle \Res_H \chi_\rho, 1\rangle_H\\
&= \frac{1}{|G|} \left\langle \Res_H \sum_\rho  d_\rho \chi_\rho, 1\right\rangle_H 
= \frac{1}{|G|} \langle \Res_H \chi_R, 1\rangle_H 
= \frac{1}{|H|}  \, . 
\end{align*}

\paragraph{The second moment}
For the second moment, observe that
\begin{align*}
\Exp_\rho X_H^2 
&= \sum_{\rho} \frac{d_\rho^2}{|G|} \frac{\langle \Res_H \chi_\rho, 1\rangle_H^2}{d_\rho^2} 
= \frac{1}{|G|}  \sum_{\rho} \langle \Res_{H \times H} \chi_{\rho \otimes \rho^*} , 1\rangle_{H \times H}\\
&= \frac{1}{|G|}  \left \langle \Res_{H \times H} \chi_B , 1\right \rangle_{H \times
H} = \frac{1}{|G|} \frac{1}{|H|^2} \sum_{h_1, h_2 \in H} |\{ g\,:\,
h_1^{-1} g h_2 = g\}| \, .
\end{align*}
Thus
\begin{align*}
\Var_\rho[X_H] &= \Exp_\rho X_H^2 - \frac{1}{|H|^2} 
= \frac{1}{|H|^2} \left( \sum_{(h_1, h_2) \in H \times H} \Pr_{g}[h_1 = g^{-1} h_2 g] \;-\; 1 \right) \\
&= \frac{1}{|H|^2} \sum_{\substack{(h_1, h_2) \in H
    \times H\\(h_1, h_2) \neq (1,1)}} \Pr_{g}[h_1 = g^{-1} h_2 g] 
= \frac{1}{|H|^2} \sum_{\substack{h \in H\\h \neq 1}} \frac{|h^G \cap H|}{|h^G|} 
\leq \frac{1}{|H|} \sum_{\substack{h \in H\\h \neq 1}} \frac{1}{|h^G|}  \, ,
\end{align*}
where $h^G = \{ g^{-1} h g : g \in G\}$ is the conjugacy class of $h$ in $G$.

Applying Chebyshev's inequality, we conclude that
\begin{equation}
\label{eq:general-restriction}
\Pr_\rho[ X_H = 0 ] 
\leq \Pr\left[ \Bigl|X_H - \frac{1}{|H|}\Bigr| \geq \frac{1}{|H|}\right] 
\leq |H| \sum_{\substack{h \in H\\h \neq 1}} \frac{1}{|h^G|} \, .
\end{equation}

Returning to the statement of the theorem, we will show the following.
\begin{theorem}
\label{thm:main}
Let $K$ be a finite nonabelian group with trivial center and let $G=K^n$.  Let $t$ be such that $2|K|^t \le \alpha \sqrt{n}$, let $h^{(1)}, \ldots, h^{(t)}$ be independent elements selected uniformly from $G$, and let $\rho = \rho_1 \otimes \cdots \otimes \rho_n$ be chosen according to the Plancherel measure on $\wG$.  Then 
  \[
  \Pr_{\rho,\{h^{(m)}\}}[ X_H =0] 
  \leq \frac{2 \alpha}{\sqrt{n}} + \left( 2^{-1/\alpha} |K|^{\alpha}\right)^{\sqrt{n}} \, ,
  \]
where $H$ is the subgroup generated by the elements $h^{(1)}, \ldots, h^{(t)}$ and $X_H$ is defined as in~\eqref{eq:xh}. When $4
|K|^t \sqrt{\log |K|} \leq \sqrt{n}$,
it
follows that with probability at least
\begin{equation}\label{eq:rep-mass}
1 - \frac{\sqrt{2}}{\sqrt[4]{n \log |K|}}
\end{equation}
a representation $\rho$ selected according to the Plancherel measure
has the property that
\begin{equation}\label{eq:rep-guarantee}
  \Pr_{\{h^{(m)}\}}\left[ \left\| \frac{1}{t} \sum_{m=1}^t \rho(h^{(m)}) \right\|_{\op} = 1\right] 
  \geq 1 - \frac{\sqrt{2}}{\sqrt[4]{n \log |K|}} \, .
  \end{equation}
\end{theorem}

Note that when $|K|^t\sqrt{\log |K|} = o(\sqrt{n})$ we may choose
$\alpha = o(1/\sqrt{\log |K|})$ in the theorem above, which guarantees that $\Pr[X_H = 0] = o(1)$.  Thus, in order to bring the expected norm of $\rho$ down to $1-\eps$ for any constant $\eps > 0$, we need $t = \Omega(\log n) = \Omega(\log \log |G|)$ random elements.  

For an element $\vec{g} = (g_1, \ldots, g_n) \in K^n$, let $\supp(\vec{g}) = |\{ i : g_i \neq 1\}|$ denote the size of the support of $\vec{g}$. We extend the definition to subgroups: for a subgroup $L < K^n$, define
$$
\supp(L) = \min_{\substack{\vec{g} \in L\\ \vec{g} \neq \vec{1}}} \supp(\vec{g}) \, . 
$$
An essential parameter of the proof below is $\supp(H)$, where $H$ is generated by the collection
$h^{(1)}, \ldots, h^{(t)}$. We may estimate this quantity by
observing that
$\supp(H) \geq \supp(\tilde{H})$ and that $\supp(\tilde{H})$ is the size of the smallest equivalence class of $\sim_{\fS}$.
Then we have the bound
\[
\supp(H) \geq d(h^{(1)}, \ldots, h^{(t)}) \, .
\]

We finally return to the proof of Theorem~\ref{thm:main}.
\begin{proof}[Proof of Theorem~\ref{thm:main}]
Let $h^{(1)}, \ldots, h^{(t)}$ be $t$ elements chosen independently and uniformly at random from $K^n$, and let $H$ denote the subgroup they generate. Additionally, let $\rho = \rho_1 \otimes \cdots \otimes \rho_n$ be a representation of $K^n$ selected according to the Plancherel measure.  Then with $d = d(h^{(1)}, \ldots, h^{(t)})$ as above,
\begin{equation}
\label{eq:pr-decomp}
\begin{split}
   \Pr_{\rho, \{h^{(m)}\}}[ X_H =0] 
   &\leq  
   \Pr_{\{ h^{(m)}\}}\left[ d < \frac{n}{2 |K|^t} \right] 
 + \Pr_{\rho, \{h^{(m)}\}}\left[ X_H = 0 \;\;\vrule\;\; d \geq \frac{n}{2 |K|^t}\right] \\
    &\leq \frac{4 |K|^{2t}}{n}  + \max_{\substack{\text{$\{ h^{(m)}\}$ such that}\\ d \geq \frack{n}{2 |K|^t}}} 
\;\Pr_\rho \left[ X_H = 0\right] \, ,
 \end{split}
 \end{equation}
the second line following from Lemma~\ref{lem:support}.

As $K$ has trivial center, all nontrivial conjugacy classes of $K$
have size at
least two.  In particular, the centralizer $Z_g = \{ h \mid
g = h^{-1} g h \}$ is a proper subgroup of $K$.  It follows that for
an element $\vec{h} \ne 1$ of $K^n$ the conjugacy class
$\vec{h}^{K^n}$ has cardinality $|\vec{h}^{K^n}| \geq
2^{\supp(\vec{h})}$.  If $d \geq \frack{n}{2 |K|^t}$, then $\supp(H)
\geq \frack{n}{2 |K|^t}$ and
 $$
 |H| \sum_{\vec{h} \neq \vec{1}} \frac{1}{|\vec{h}^{K^n}|} 
 \leq |H|^2 \,2^{-\supp(H)} 
 \leq |K|^{2|K|^t} 2^{-n/2|K|^t} \, ,
 $$
 by Lemma~\ref{lem:subgroup-size}.  Writing $2|K|^t \le \alpha \sqrt{n}$, we have
 $$
 |H| \sum_{\vec{h} \neq \vec{1}} \frac{1}{|\vec{h}^{K^n}|} 
 \leq \left( 2^{-1/\alpha} |K|^\alpha \right)^{\sqrt{n}}  \, . 
 $$
Combining equations~\eqref{eq:general-restriction} and~\eqref{eq:pr-decomp} gives 
\[
\Pr_{\rho, \{h^{(m)}\}}[ X_H =0] 
\leq \frac{2 \alpha}{\sqrt{n}} + \left( 2^{-1/\alpha} |K|^{\alpha}\right)^{\sqrt{n}} \, ,
\]
completing the proof. The bounds of~\eqref{eq:rep-mass}
and~\eqref{eq:rep-guarantee} are achieved by setting $\alpha =
1/(2\sqrt{\log |K|})$, in which case
\[
\frac{2 \alpha}{\sqrt{n}} + \left( 2^{-1/\alpha}
  |K|^{\alpha}\right)^{\sqrt{n}} \leq \frac{1}{\sqrt{n\log |K|}} +
\left(\frac{1}{2\sqrt{2}}\right)^{\sqrt{n \log |K|}}  \leq \frac{2}{\sqrt{n\log |K|}} \, .\qedhere
\]
\end{proof}

\paragraph{Acknowledgments}  S.L. is supported by NSF CAREER award 1350481, and C.M. and A.R. are supported by NSF grant CCF-1247081.

\end{document}